\documentclass{scrartcl}

\usepackage[english]{babel}
\usepackage[utf8]{inputenc}
\usepackage[T1]{fontenc}
\usepackage{amsmath,amsfonts,amsthm,amssymb}
\usepackage{csquotes}
\usepackage{pgf}
\usepackage{xspace}
\usepackage{chngcntr}
\usepackage{xcolor}
\usepackage{units}

\definecolor{darkgreen}{RGB}{0,100,0}

\newcommand*{\df}[1]{\textit{#1}}
\newcommand{\textenum}[2]{$#1$, $\ldots$, $#2$}

\DeclareMathOperator{\IV}{Int}
\newcommand{\Int}[1]{\ensuremath{\IV(#1)}\xspace}

\newcommand{\IntA}{\ensuremath{\IV(A,\Matrn{D})}\xspace}
\DeclareMathOperator{\IntIm}{Int-Im}
\newcommand{\IntImA}{\ensuremath{\IntIm(A,\Matrn{D}) }\xspace}

\newcommand*{\FacRing}[3]{\ensuremath{\nicefrac{#1}{#2^{#3}{#1}}}\xspace}
\newcommand{\Res}[2]{\ensuremath{[{#1}]_{#2}}\xspace}
\newcommand{\Resp}[1]{\ensuremath{[{#1}]_{p}}\xspace}
\newcommand*{\cp}[1]{\ensuremath{\widehat{#1}}\xspace}
\newcommand{\Rl}{\ensuremath{R_{\ell}}\xspace}
\newcommand{\Loc}[2]{\ensuremath{{#1}_{(#2)}}\xspace}

\newcommand{\ModuleVar}[2]{\ensuremath{{\left\langle #2 \right\rangle}_{#1}}\xspace}
\newcommand{\basis}[2]{\ensuremath{\mathbf{#1}_{#2}}\xspace}
\newcommand*{\Nl}{\ensuremath{\mathbf{0}}\xspace}

\DeclareMathOperator{\lc}{\mathsf{lc}}
\DeclareMathOperator{\val}{\mathsf{w}}

\renewcommand{\mod}{\ensuremath{\;\operatorname{mod}\;}}
\theoremstyle{plain}
\newtheorem{Thm}{Theorem}[section]

\newtheorem{Lem}[Thm]{Lemma}
\newtheorem{Pro}[Thm]{Proposition}
\newtheorem{Cor}[Thm]{Corollary}

\theoremstyle{definition}
\newtheorem{Def}[Thm]{Definition}

\newtheorem{NotConv}[Thm]{Notation and Conventions}

\newtheorem{Exa}[Thm]{Example}

\newtheorem{Rem}[Thm]{Remark}

\theoremstyle{remark}
\newtheorem*{Claim}{Claim} 

\newcommand*{\N}{\ensuremath{\mathbb{N}}\xspace}

\newcommand*{\Z}{\ensuremath{\mathbb{Z}}\xspace}
\newcommand*{\Q}{\ensuremath{\mathbb{Q}}\xspace}

\newcommand{\Primes}{\ensuremath{\mathbb{P}}\xspace}

\newcommand{\NullI}[1][]{\ensuremath{\mathsf{N}^{#1}}\xspace}
\newcommand{\NullIovD}[2][]{\ensuremath{\mathsf{N}^{#1}_{#2}}\xspace}

\newcommand{\SIndexSet}[2]{\ensuremath{\mathcal{I}_{#2}^{\star}}\xspace}
\newcommand{\IndexSet}[2]{\ensuremath{\mathcal{I}_{#2}}\xspace}

\newcommand{\pdeg}{\ensuremath{\mathsf{d}_p}\xspace}

\DeclareMathOperator{\successor}{succ}
\newcommand{\nf}[1]{\ensuremath{\successor(#1)}\xspace}

\DeclareMathOperator{\Matrices}{M}
\DeclareMathOperator{\GenLin}{GL}

\newcommand{\Matr}[2]{\ensuremath{\Matrices_{#1}(#2)}\xspace}
\newcommand{\Matrn}[1]{\ensuremath{\Matrices_n(#1)}\xspace}
\newcommand{\GLn}[1]{\ensuremath{\GenLin_n(#1)}\xspace}

\DeclareMathOperator{\companion}{\mathcal{C}}
\newcommand{\Comp}[1]{\ensuremath{\companion_{#1}}\xspace}

\DeclareMathOperator{\diag}{diag}

\usepackage[
pdftitle={Null ideals of matrices over residue class rings of principal ideal domains},
pdfsubject={commutative ring theory},
pdfauthor={Roswitha Rissner},
pdfkeywords={null ideal, minimal polynomial, characteristic polynomial, integer-valued polynomials, commutative rings},
bookmarksopen=true,
bookmarksnumbered=true,
pdfpagemode=UseOutlines,
pdfdisplaydoctitle=true,
colorlinks=true,
linkcolor=black,
citecolor=darkgreen,
urlcolor=darkgreen
]{hyperref}

\numberwithin{equation}{section}
\counterwithin{table}{section}
\begin{document}
\title{Null ideals of matrices over residue class rings of principal ideal domains}
\author{
Roswitha Rissner\footnotemark[1]
}
\maketitle
\def\thefootnote{\fnsymbol{footnote}}
\footnotetext[1]{Graz University of Technology, Department of Analysis and Computational Number Theory, Steyrergasse 30, 8045 Graz,
Austria. \texttt{rissner@math.tugraz.at}. Supported by Austrian Science Fund (FWF): P23245-N18  }

\begin{abstract}%
\noindent Given a square matrix $A$ with entries in a commutative ring $S$,
the ideal of $S[X]$ consisting of polynomials $f$ with $f(A) =0$
is called the null ideal of $A$. 
Very little is known about null ideals of matrices over general commutative rings. 
First, we determine a certain generating set of the null ideal of a matrix in case $S = \FacRing{D}{d}{}$ is the residue class ring of 
a principal ideal domain $D$ modulo $d\in D$. After that we discuss two applications.
We compute a decomposition of the $S$-module $S[A]$ into cyclic $S$-modules and explain the strong 
relationship between this decomposition and the determined generating set of the null ideal of $A$. 
And finally, we give a rather explicit description of the ring \IntA of all integer-valued polynomials on $A$.


\smallskip
\noindent \textbf{2010 Math. Subj. Class.} 11C08, 11C20, 13F20, 15A15, 15B33, 15B36 

\end{abstract}

\section{Introduction}

Matrices with entries in commutative rings arise in numerous contexts, both in pure and applied mathematics.  
However, many of the well-known results of classical linear algebra do not hold in this general setting. This is the case even if the underlying ring is a domain (but not a field). For a general introduction to matrix theory over commutative rings we refer to the textbook of Brown~\cite{Brown1993}.

The purpose of this paper is to provide a better understanding of null ideals of square matrices over residue class rings of principal ideal domains. 
\begin{Def}
Let $S$ be a commutative ring, $A\in \Matrn{S}$ an $n$$\times$$n$-square matrix $A$ over $S$. 
The \df{null ideal} $\NullI[S](A)$ of $A$ (over $S$) is the set of all polynomials which annihilate $A$, that is, 
\begin{align*}
 \NullI[S](A) =  \{\,f \in S[X] \mid f(A) = 0\,\}.
\end{align*}
We often write $\NullI(A)$ instead of $\NullI[S](A)$ if the underlying ring is clear from the context.
\end{Def}
In case $S$ is a field, it is well-known that the null ideal of $A$ is generated by a uniquely determined monic polynomial, the so-called \df{minimal polynomial} $\mu_A$ of $A$.  Further, it is known  that if $S$ is a domain, then the null ideal of every square matrix is principal (generated by $\mu_A$) if and only if $S$ is integrally closed, (Brown~\cite{Brown98}, Frisch~\cite{Frisch2004}). 
However, little is known about the null ideal of a matrix with entries in a commutative ring. 
The well-known Cayley-Hamilton Theorem states that every square matrix over a commutative ring satisfies its own characteristic equation (cf.~\cite[Theorem~XIV.3.1]{Lang2002}). Therefore there always exists a monic polynomial in $S[X]$ of minimal degree which annihilates the matrix. 
\begin{Def}\label{def:minpoly}
 Let $A\in\Matrn{S}$ be a square matrix over a commutative ring $S$. If $f\in S[X]$ is a monic  polynomial with $f(A) = 0$ and there exists no monic polynomial in $S[X]$ of smaller degree with this property, then we call $f$ a \df{minimal polynomial} of $A$ over $S$. 
\end{Def}
%
Note that, in case $S$ is a field, the definition above is consistent with the classical definition of the (uniquely determined) minimal polynomial of a square matrix. However in general, if $S$ is not a field, a minimal polynomial of a matrix over $S$ is not uniquely determined, although its degree is. It is known that if $S$ is a domain, then the null ideal of $A$ is principal if and only if $A$ has a uniquely determined minimal polynomial over $S$, which is in turn equivalent to  the (uniquely determined) minimal polynomial $\mu_A$ of $A$ over the quotient field of $S$ being in $S[X]$. 

Brown discusses conditions for the null ideal to be principal over a general commutative ring $R$ (with identity). 
In~\cite{Brown2005}, he gives sufficient conditions on certain $R[X]$-submodules of the null ideal for the null ideal to be principal. 
There is also earlier work of Brown investigating the relationship of the null ideals of certain pairs of square matrices over a commutative ring (which he refers to as spanning rank partners), see~\cite{Brown98}, \cite{Brown99}. 

A better understanding of null ideals of matrices over residue class rings of domains has applications in the theory of integer-valued polynomials on matrix rings. Let $D$ be a domain with quotient field $K$, and let $A\in \Matrn{D}$. 
For a polynomial $f\in K[X]$, the image $f(A)$ of $A$ under $f$ is a matrix with entries in $K$. There are two immediate questions in this context:
For which $f\in K[X]$ does $f(A)\in\Matrn{D}$ hold? And what are the images of $A$ under these polynomials?
We set 
\begin{align*}
 \IntA = \{\,f\in K[X]\mid f(A)\in\Matrn{D}\,\}
\end{align*}
the ring of integer-valued polynomials on $A$, and we denote by  
\begin{align*}
 \IntImA = \{\,f(A) \mid f \in \IntA\,\}
\end{align*}
the ring of images of $A$ under integer-valued polynomials of $A$. 
\IntA is an overring of the ring of integer-valued polynomials on the $D$-algebra \Matrn{D}, that is,
\begin{align*}
\Int{\Matrn{D}} = \{\,f\in K[X]\mid f(\Matrn{D}) \subseteq \Matrn{D}\,\}.
\end{align*}
The ring \Int{\Matrn{D}} and other generalizations of integer-valued polynomial rings are subject of recent research, see \cite{Evrard2013}, \cite{Frisch2010}, \cite{Frisch2013},   \cite{LoperWerner2012}, \cite{Peruginelli2014} and \cite{PeruginelliWerner2014}.

The connection between integer-valued polynomials on a matrix and null ideals of matrices is the following: Let $f\in K[X]$, then there exist $g\in D[X]$ and $d\in D$ such that $f = g/d$. The following assertion holds:
\begin{align*}
 \forall\, d\in D\setminus \{0\}\;\;\forall\, g\in D[X]\;:\left( \frac{g}{d} \in \IntA   \; \Longleftrightarrow  \; g(A) \equiv 0 \mod d\Matrn{D} \right)
\end{align*}
which is the case if and only if the residue class of $g$ is in the null ideal of $A$ over the residue class ring $\FacRing{D}{d}{}$.

In this paper, we investigate the null ideal of a square matrix $A$ over the residue class ring \FacRing{D}{d}{} of a principal ideal domain $D$ modulo $d\in D$. In Section~\ref{sec:GeneratorsNullIdeal} we provide a description of a specific set of generators of the null ideal of a matrix with entries in \FacRing{D}{d}{}. 
With this goal in mind, we generalize the notion of the null ideal at the beginning of the section. 
Instead of looking only at the ideal of polynomials which map $A$ to the zero ideal, we are also interested in those polynomials which map $A$ to the ideal $d\Matrn{D}$, cf.~Definition~\ref{def:Jideal}. This point of view has the advantage that it allows us to work over domains instead of residue class rings (which, in general, have zero-divisors). Further, it turns out that it suffices to consider the special case when $d=p^{\ell}$ is a prime power ($\ell\in\N$ and $p\in D$ a prime element). The main result of this section is Theorem~\ref{thm:MinGenNullI} which describes a specific set of generators of the null ideal of a matrix over $\FacRing{D}{p}{\ell}$. 
However, this description is theoretic; so far, we do not know how to determine them algorithmically in general. It is possible to compute these generators explicitly in case of diagonal matrices. We present this approach at the end of Section~\ref{sec:GeneratorsNullIdeal}. 

The theoretical results in Section~\ref{sec:GeneratorsNullIdeal} allow us to present two applications.
In Section~\ref{sec:modulestructure} we analyze the \FacRing{D}{p}{\ell}-module structure of $\FacRing{D}{p}{\ell}[A]$ for $A \in \Matrn{\FacRing{D}{p}{\ell}}$. 
As a finitely generated module over a principal ideal ring, $\FacRing{D}{p}{\ell}[A]$ decomposes into a direct sum of cyclic submodules with uniquely determined invariant factors,  according to~\cite[Theorem~15.33]{Brown1993}. We describe this decomposition explicitly and find a strong relationship to the generating set of $\NullI[\FacRing{D}{p}{\ell}](A)$ from Section~\ref{sec:GeneratorsNullIdeal}. This allows us to find certain invariant properties of this generating set.

In the last section we apply the knowledge about the null ideal gained in Section~\ref{sec:GeneratorsNullIdeal} to integer-valued polynomials. We give an explicit description of the ring \IntA using the generating set of the null ideal of $A$ modulo finitely many prime powers $p^{\ell}$. Once this description is given, the ring \IntImA of images of 
$A$ under integer-valued polynomials is easily determined.

\section{Generators of the null ideal}
\label{sec:GeneratorsNullIdeal}

As already mentioned in the introduction, the goal of this section is to compute a generating set of the null ideal of a square matrix over residue class rings of  a principal ideal domain $D$. However, as it is much more convenient to work over domains instead of residue class rings (which, in general, contain zero-divisors) it turns out to be useful to generalize the notion of the null ideal of a matrix. Instead of investigating only ideals of polynomials which map a given matrix to the zero ideal, we are also interested in polynomials which map the matrix to the ideal $J\Matrn{D}$ where $J$ is an ideal of $D$. Although the results in this paper are restricted to matrices over principal ideal domains and their residue class rings, the following definitions make sense in much broader generality. Therefore, up to and including Remark~\ref{rem:01mp}, we allow the underlying ring to be a general commutative ring.

\begin{Def}\label{def:Jideal}
 Let $S$ be a commutative ring, $J$ an ideal of $S$ and $A\in \Matrn{S}$ a square matrix. We call
 \begin{align*}
  \NullIovD[S]{J}(A) = \{\,f\in S[X] \mid f(A) \in J\Matrn{S}\,\}
 \end{align*}
 the \df{$J$-ideal of $A$} (over $S$).  Further, we say $f$ is a \df{$J$-minimal polynomial of $A$} (over $S$), 
 if $f$ is a monic polynomial in $\NullIovD[S]{J}(A)$ and $\deg(f) \leq \deg(g)$ for all monic polynomials $g\in  \NullIovD[S]{J}(A)$.
If the underlying ring is clear from the context, we often omit the superscript and  write $\NullIovD{J}(A)$ instead of $\NullIovD[S]{J}(A)$.
\end{Def}

\begin{Rem}\label{rem:nullidealnewdef}
  With this definition, the null ideal $\NullI[S](A)$ of $A$ is just the \Nl-ideal $\NullIovD[S]{\Nl}(A)$ (that is if $J=\Nl$ is the trivial ideal). Further, the \Nl-minimal polynomials of a matrix $A$ are exactly the minimal polynomials of $A$ over $S$, cf.~Definition~\ref{def:minpoly}. We often use the more classical notation $\NullI[S](A)$ (and say minimal polynomial instead of \Nl-minimal polynomial) as it is less technical.
\end{Rem}

For the remainder of this paper, let the following notation and conventions hold.
\begin{NotConv}
 Let $S$ be a commutative ring, $J$ an ideal of $S$ and $A\in \Matrn{S}$. We identify the isomorphic rings $\Matrn{\nicefrac{S}{J}} = \nicefrac{\Matrn{S}}{J\Matrn{S}}$ and $\nicefrac{S}{J}[X] = \nicefrac{S[X]}{JS[X]}$ and write \Res{\,.\,}{J} to denote residue classes modulo $J$. 
\end{NotConv}

\begin{Rem}
 The null ideal $\NullI[\nicefrac{S}{J}](\Res{A}{J})$ of the residue class $\Res{A}{J} \in \Matrn{\nicefrac{S}{J}}$ of $A$ modulo $J$ is the image of the $J$-ideal $\NullIovD[S]{J}(A)$ of $A$ under the projection modulo $J$, that is, 
\begin{align*}
 \NullI[\nicefrac{S}{J}](\Res{A}{J}) = \NullIovD[\nicefrac{S}{J}]{\Nl}(\Res{A}{J})=\{\,\Res{f}{J}\in \nicefrac{S}{J}[X] \mid f\in \NullIovD[S]{J}(A)\,\}.
\end{align*}
\end{Rem}

\begin{Rem}\label{rem:defminpoly_ovD}
 Whether a monic polynomial $f\in S[X]$ is a $J$-minimal polynomial of $A$ depends only on the residue class of $A$ modulo~$J$.
 If $J\neq S$ is a proper ideal, then a monic polynomial $f\in S[X]$ is a $J$-minimal polynomial if and only if its residue class $\Res{f}{J}\in \nicefrac{S}{J}[X]$ is a $\Nl$-minimal polynomial of $\Res{A}{J}$ over $\nicefrac{S}{J}$. (In case $J=S$, one would have to think about the meaning of  ``monic'' polynomial over the null ring to state a similar result. As we do not want to consider the zero polynomial to be monic, we exclude this case.)
 Further, let $I$ be an ideal of $S$ such that $I\subseteq J$. Then $\nicefrac{S}{J} \simeq \nicefrac{(S/I)}{(J/I)}$. Therefore,
 $f$ is a $J$-minimal polynomial of $A$ over $S$ if and only if $\Res{f}{I} \in \nicefrac{S}{I}[X]$ is a $\nicefrac{J}{I}$-minimal polynomial of $\Res{A}{I}$ over $\nicefrac{S}{I}$. 
\end{Rem}

\begin{Rem}\label{rem:01mp} 
The $S$-ideal $\NullIovD[S]{S}(A)$ of every square matrix $A$ over $S$ is just the whole ring $S[X]$ (that is, if $J=(1)=S$ is the unit ideal). It is therefore generated by the constant polynomial $1$. Hence the constant $1$ is the (uniquely determined) $S$-minimal polynomial of every square matrix $A$ over $S$.
\end{Rem}

As stated at the beginning of this section, for the remainder of this paper we restrict the underlying ring to be a principal ideal domain. 
Hence, from this point on, the following notation and conventions hold.

\begin{NotConv}\label{conv:pid}
Let $D$ be a principal ideal domain and \Primes be a complete set of representatives of associate classes of prime elements of $D$. Note that $J=(d)$ for some $d\in D$.  We write $\NullIovD{d}(A)$ instead of $\NullIovD{(d)}(A)$ (and omit the superscript $D$). For the residue classes modulo $d$, we often write \Res{\,.\,}{d} instead of \Res{\,.\,}{(d)}.
\end{NotConv}

The first result of this section is the following lemma. It states a simple but crucial relation between the degrees and the leading coefficients of polynomials in the $(d)$-ideal of a matrix. 
Observe that if the leading coefficient of a polynomial $g\in D[X]$ (denoted by $\lc(g)$) is coprime to $d$, then it is a unit modulo~$d$. Hence, there exists an element $c\in D$ such that $\Res{cg}{d}$ is a monic polynomial in $\FacRing{D}{d}{}[X]$. In particular, this implies the following lemma.

\begin{Lem}\label{lem:leadingcoeffanddegree}
 Let $D$ be a principal ideal domain and $d\in D$ with $d\notin \{0,1\}$. If $f\in D[X]$ is a $(d)$-minimal polynomial, then all polynomials $g\in \NullIovD{d}(A)$ with $\deg(g) < \deg(f)$ have a leading coefficient $\lc(g)$ which is not invertible modulo $d$, that is, $\gcd(\lc(g),d) \neq 1$.
\end{Lem}

Recall that $\NullIovD{0}(A) = \NullI(A)$ is the null ideal of $A$ over $D$. 
Further, $D$ is integrally closed, since it is a principal ideal domain. As mentioned in the introduction, this implies that the minimal polynomial of every square matrix in $\Matrn{D}$ is in $D[X]$ and generates its null ideal. 
In particular, 
\begin{align*}
\NullIovD{0}(A) = \NullI(A)=  \mu_AD[X] 
\end{align*}
holds, where $\mu_A\in D[X]$ is the minimal polynomial of $A$ over $K$. This completes the case $d=0$.
For $d\neq 0$, we first observe, that it suffices to compute $\NullIovD{d}(A)$ for $d=p^{\ell}$ with $p\in D$ a prime element and $\ell\in \N$.

\begin{Lem}\label{lem:CRT>primepowers}
 Let $D$ be a principal ideal domain, $A\in \Matrn{D}$ and $a,b\in D$ be coprime elements. Then 
 \begin{align*}
   \NullIovD{ab}(A) = a\,\NullIovD{b}(A) + b\,\NullIovD{a}(A).
 \end{align*}
\end{Lem}
\begin{proof}
The inclusion ``$\supseteq$'' is trivial.
For ``$\subseteq$'', let $g\in\NullIovD{ab}(A)$. Since $a$ and $b$ are coprime, there exist $h_1,h_2\in D[X]$ such that 
\begin{align*}
 g = a h_1 + b h_2.
\end{align*}
Then
\begin{align*}
 a h_1 (A) &= g(A) - b h_2(A) \in b\Matrn{D} \text{ and } \\
 b h_2 (A) &= g(A) - a h_1(A) \in a\Matrn{D}.
\end{align*}
It follows that $h_1 \in \NullIovD{b}(A)$ and $h_2 \in \NullIovD{a}(A)$, which completes the proof.
\end{proof}

\begin{NotConv}
For the rest of this section we fix the prime element $p\in D$. If $A\in \Matrn{D}$ is fixed, we often write $\NullIovD{p^{\ell}}$ instead of  $\NullIovD{p^{\ell}}(A)$.
\end{NotConv}

Our goal is to determine polynomials $f_0,\ldots,f_m\in D[X]$ such that 
\begin{align*}
 \NullIovD{p^{\ell}}(A) = \{\,f \in D[X] \mid f(A) \equiv 0 \bmod p^{\ell}\,\} 
                 = \sum_{i=0}^m f_i D[X] 
\end{align*}
for $A\in \Matrn{D}$.
Since $\FacRing{D}{p}{}$ is a field, the null ideal of $A$ modulo $p$ is a principal ideal. Hence 
\begin{align*}
\NullIovD{p}(A) = \nu_1D[X] + pD[X] 
\end{align*}
where $\nu_1$ is a $(p)$-minimal polynomial of $A$.
The degree of $\nu_1$ is, by definition, independent of the choice of a $(p)$-minimal polynomial. 
\begin{Def}\label{def:pdeg}
 Let $\nu_1 \in D[X]$ be a $(p)$-minimal polynomial $A$. We call $\pdeg(A) = \deg(\nu_1)$ \df{the $p$-degree of $A$} and write $\pdeg$ if the matrix is clear from the context.
\end{Def}
Note again, that this definition depends only on the residue class of $A$ modulo $p$, cf.~Remark~\ref{rem:defminpoly_ovD}.
Observe that the following inclusions hold
\begin{align*}
\resizebox{.98\hsize}{!}{%
$
\mu_AD[X] 
= \NullI(A) 
= \NullIovD{0} \subseteq \cdots \subseteq \NullIovD{p^{\ell}} \subseteq \NullIovD{p^{\ell-1}} \subseteq \cdots \subseteq \NullIovD{p} = \nu_1D[X]+pD[X] 
\subseteq  D[X] = \NullIovD{1} 
$
}
\end{align*}
where $\nu_1$ is a $(p)$-minimal polynomial of $A$. The $p$-degree of $A$ is a lower bound for the degree of all polynomials in $\NullIovD{p^{\ell}}\setminus p^{\ell}D[X]$, as the following lemma states.

\begin{Lem}\label{lem:degree>=k}
Let $D$ be a principal ideal domain, $\ell \geq 1$ and $A\in\Matr{n}{D}$. If $f\in \NullIovD{p^{\ell}}(A)\setminus p^{\ell}D[X]$, then $\deg(f) \geq \pdeg(A)$.
\end{Lem}
\begin{proof}
We prove this by contradiction. Let $\ell\geq 1$ be minimal such that there exists a polynomial $f\in \NullIovD{p^{\ell}}\setminus p^{\ell}D[X]$ with $\deg(f) < \pdeg$. Without restriction, we choose $f$ to be a polynomial of minimal degree with this property, that is, if $g\in \NullIovD{p^{\ell}}$ with $\deg(g) < \deg(f)$, then $g\in p^{\ell}D[X]$.

If $\ell= 1$, then $p$ divides $\lc(f)$ according to Lemma~\ref{lem:leadingcoeffanddegree}. Hence $f'=\lc(f)X^{\deg(f)}\in pD[X]\subseteq \NullIovD{p}$, and therefore $f-f'\in \NullIovD{p}$ is a polynomial with degree strictly smaller than $\deg(f)$.   Therefore $f-f'\in pD[X]$ which implies $f\in pD[X]$, a contradiction. 

Hence $\ell>1$, and since $f\in \NullIovD{p^{\ell}}$ it follows that $f\in \NullIovD{p^{\ell-1}}$. Then, due to the minimality of $\ell$, it follows that $f\in p^{\ell-1}D[X]$. Let $h\in D[X]$ such that $f=p^{\ell-1}h$. Then $\deg(h) = \deg(f) <\pdeg$ and 
\begin{align*}
f(A) = p^{\ell-1}h(A) \equiv 0 \mod p^{\ell} 
\end{align*}
which is equivalent to $h\in \NullIovD{p}$. Then again, by minimality of  $\ell>1$, it follows that $h\in pD[X]$ and therefore $f\in p^{\ell}D[X]$, contrary to our assumption.

\end{proof}
%

The next proposition provides one of the main tools in this section. 
It states a simple but important result, which allows us to deduce various properties of the generators of \NullIovD{p^{\ell}}. 

\begin{Pro}\label{prop:decompositionofnullpolys}
Let $D$ be a principal ideal domain, $p\in D$ a prime element. Further, let $A\in\Matr{n}{D}$ be a square matrix over $D$, and $\nu_{\ell}$ be a $(p^{\ell})$-minimal polynomial of $A$ (for $\ell \geq 1$). If
$f\in \NullIovD{p^{\ell}}(A)$, then there exist uniquely determined polynomials $q,g\in D[X]$ such that $\deg(g) < \deg(\nu_{\ell})$ and
\begin{align*}
 f = q\nu_{\ell} + pg .
\end{align*}
In particular, 
\begin{align*}
 \NullIovD{p^{\ell}}(A) = \nu_{\ell}D[X] + p\,\NullIovD{p^{\ell-1}}(A).
\end{align*}
\end{Pro}

\begin{proof}
Let $f\in \NullIovD{p^{\ell}}$. Since $\nu_{\ell}$ is monic for every $\ell \geq 1$, 
we can use polynomial division: there exist uniquely determined $q,r\in D[X]$ with $\deg(r) < \deg(\nu_{\ell})$ such that 
\begin{align}\label{prf:polydiv}
 f = q\nu_{\ell} +r.
\end{align}
It is easily seen that $r\in \NullIovD{p^{\ell}}$, hence it suffices to prove the following claim.

\begin{Claim}
 Let $r\in \NullIovD{p^{\ell}}$ with $\deg(r)<\deg(\nu_{\ell})$. Then $r\in pD[X]$.
\end{Claim}

If $\ell=1$, then the assertion follows from Lemma~\ref{lem:degree>=k}.
Let $\ell>1$ be minimal such that the claim is false. Further, choose $r\in \NullIovD{p^{\ell}}$ with $\deg(r)<\deg(\nu_{\ell})$ of minimal degree such that $r\notin pD[X]$.  Since $r\in \NullIovD{p^{\ell}}$ it is in $\NullIovD{p^{\ell-1}}$ too.  By minimality of $\ell$, there exist $q',g'\in D[X]$ such that 
\begin{align*}
  r = q'\nu_{\ell-1} + pg' 
\end{align*}
with $\deg (g') < \deg(\nu_{\ell -1})$. Since $r\notin pD[X]$, it follows that $q'\notin pD[X]$. Therefore, there exists $q_1,q_2\in D[X]$ with $q_2\neq 0$ and no non-zero coefficient of $q_2$ is divisible by $p$ such that 
\begin{align*}
 q' = p q_1 +q_2.
\end{align*}
Hence $r$ can be written in the following form
\begin{align*}
 r = q_1\underbrace{p\nu_{\ell-1}}_{\in \NullIovD{p^{\ell}}} + q_2\nu_{\ell-1} +pg' \in \NullIovD{p^{\ell}}.
\end{align*}
This, however, implies that $f' = q_2\nu_{\ell-1} +pg' \in \NullIovD{p^{\ell}}$. 
Observe, that $\deg(g') < \deg(\nu_{\ell-1})$ which implies that $\lc(f') = \lc(q_2)\lc(\nu_{\ell-1}) = \lc(q_2)$ is not divisible by $p$. On the other hand, 
\begin{align*}
\deg(f') = \deg(q_2)+\deg(\nu_{\ell-1}) \leq  \deg(r)< \deg(\nu_{\ell}) 
\end{align*}
which implies, by Lemma~\ref{lem:leadingcoeffanddegree}, that $p$ divides $\lc(f')$, a contradiction.
\end{proof}

We state a corollary of Proposition~\ref{prop:decompositionofnullpolys}, which is particularly useful: the smaller the degree of a polynomial in \NullIovD{p^{\ell}}, the higher the power of $p$ that divides it. 
\begin{Cor}\label{cor:degreeandppower}
 Let $D$ be a principal ideal domain and $p\in D$ a prime element. Further, let $A\in \Matr{n}{D}$, $\ell\geq 1$, and $\nu_j$ be $(p^j)$-minimal polynomials of $A$ for $1\leq j\leq \ell$. If $f\in \NullIovD{p^{\ell}}(A)$, then  
\begin{align*}
 \deg (f) < \deg(\nu_j) \quad \Longrightarrow \quad f \in p^{\ell-(j-1)}D[X].  
\end{align*}
In particular, if $\deg(\nu_{\ell}) = \deg(\nu_j)$, then 
\begin{align*}
 \NullIovD{p^{\ell}}(A) = \nu_{\ell}D[X] + p^{\ell-(j-1)}\NullIovD{p^{j-1}}(A) 
\end{align*}
holds.
\end{Cor}
\begin{proof}
We use induction on $\ell\geq 1$. Let $f\in \NullIovD{p^{\ell}}$ with $\deg (f) < \deg(\nu_j)\leq \deg(\nu_{\ell})$. Observe, that $f = pg$ for some $g\in \NullIovD{p^{\ell-1}}$, according to Proposition~\ref{prop:decompositionofnullpolys}. 
Hence if $\ell=j\geq 1$, then the assertion follows. In particular, if $\ell=1$, then $j=1$ which proves the basis.

Hence assume $\ell >j>1$. Then $j\leq \ell-1$ and we can apply the induction hypothesis to $g\in \NullIovD{p^{\ell-1}}$ and conclude that $g\in p^{\ell-1-(j-1)}D[X]$ which completes the proof.
\end{proof}
At this point, we have enough tools to prove that the polynomials $p^{\ell-i}\nu_i$ generate \NullIovD{p^{\ell}}.

Recall that $\NullIovD{1}(A) = D[X]$ is generated by the constant polynomial $1$ (see~Remark~\ref{rem:01mp}). Therefore the constant polynomial $\nu_0 = 1$ is the (uniquely determined) $(p^0)$-minimal polynomial of $A$ for all prime elements $p$. 

Again, we use induction on $\ell$ and $\NullIovD{1}(A)=\NullIovD{p^0}(A) = D[X] = p^0\nu_0D[X]$ serves as induction basis. The induction step is an application of Proposition~\ref{prop:decompositionofnullpolys}. 
%

\begin{Thm}\label{thm:NullI}
Let $D$ be a principal ideal domain and $p\in D$ a prime element. Further, let $A\in \Matr{n}{D}$ be a square matrix over $D$, $\ell\geq 0$, and $\nu_j\in D[X]$ be $(p^j)$-minimal polynomials of $A$ for $0\leq j\leq \ell$. Then 
\begin{align*}
 \NullIovD{p^{\ell}}(A) =\sum_{j=0}^{\ell} p^{\ell-j}\nu_j D[X].
\end{align*}
\begin{flushright}
 \qed
\end{flushright}

\end{Thm}

%
Theorem~\ref{thm:NullI} states that the null ideal \NullIovD{p^{\ell}} of $A$  is generated by the $\ell +1$ polynomials $p^{\ell-i}\nu_i$  for $0\leq i \leq \ell$. However, in general this is not a minimal generating set. While we are not able to decide which subsets are minimal generating sets, we can still identify some redundant polynomials in $\{\,p^{\ell-i}\nu_i \mid 0\leq i \leq \ell\,\}$.
Note that $\deg(\nu_{i+1}) \geq \deg(\nu_{i})$ holds for all $i \geq 0$. It turns out that it suffices to keep one polynomial of each degree in $\{\,\deg(\nu_i) \mid 0\leq i \leq \ell\,\}$ to generate \NullIovD{p^{\ell}}. 
Theorem~\ref{thm:MinGenNullI} states explicitly, which subsets of $\{\,p^{\ell-i}\nu_i \mid 0\leq i \leq \ell\,\}$ we might choose.
Although the resulting generating set might still not be minimal, it is strongly connected to a certain decomposition of $\FacRing{D}{p}{\ell}[\Res{A}{d}]$ into cyclic $\FacRing{D}{p}{\ell}$-submodules which is the topic of Section~\ref{sec:modulestructure}.

Theorem~\ref{thm:NullI} and Corollary~\ref{cor:degreeandppower} imply that, if $\deg(\nu_{j+1})=\deg(\nu_{j})$ for some $0\leq j< \ell$, then $\NullIovD{p^{\ell}}$ is generated by $\{\,p^{\ell-i}\nu_i\mid 0\leq i\leq \ell\,\}\setminus \{p^{\ell-j}\nu_j\}$, cf.~Theorem~\ref{thm:MinGenNullI} below. 
For each $d\in  \{\,\deg(\nu_i)\mid 0\leq i\leq \ell\,\}$ we want to keep only the largest $j$ such that $\deg(\nu_j) = d$.
This motivates the following definition.
%
\begin{Def}\label{def:indexset}
 Let $A\in \Matrn{D}$ be a square matrix with $(p^i)$-minimal polynomials for $1\leq i \leq \ell$. Then we call 
\begin{align*}
\IndexSet{A}{\ell} = \{\ell\} \cup \{\,i \mid 0\leq i < \ell, \deg(\nu_i) < \deg(\nu_{i+1})\,\} 
\end{align*}
the \df{$\ell$-th index set of $A$ (with respect to the prime element $p$)}.
\end{Def}
\begin{Rem}\label{rem:index_set_resclass}
The (uniquely determined) degree of a $(p^j)$-minimal polynomial of $A$ depends only on the residue class of $A$ modulo $p^{\ell}$, not on the choice of a representative.
\end{Rem}

\begin{Rem}\label{rem:propofellindexset}
 The indices $0$ and $\ell$ are always contained in \IndexSet{A}{\ell}. Further, 
 the $\ell$-th index set \IndexSet{A}{\ell} of $A$ satisfies the following:
\begin{enumerate}
 \item If $\deg \nu_{\ell} \neq \deg \nu_{\ell-1}$, then $\IndexSet{A}{\ell} = \{\ell\} \cup \IndexSet{A}{\ell - 1}$.
 \item If $\deg \nu_{\ell} = \deg \nu_{\ell-1}$, then $\IndexSet{A}{\ell} = \{\ell\} \cup (\IndexSet{A}{\ell-1} \setminus \{\ell-1\})$.
\end{enumerate}
\end{Rem}

The $\ell$-th index set of $A$ contains the  information which  $(p^j)$-minimal polynomials we need to generate \NullIovD{p^{\ell}} as stated by the next theorem.

\begin{Thm}\label{thm:MinGenNullI}
Let $D$ be a principal ideal domain, $p\in D$ a prime element and $\ell\geq 0$. Further, let $A\in \Matr{n}{D}$ be a square matrix over $D$ with $\ell$-th index set $\IndexSet{A}{\ell}$  and $\nu_i\in D[X]$ be $(p^i)$-minimal polynomials for $0\leq i\leq \ell$. Then 
\begin{align*}
 \NullIovD{p^{\ell}}(A) =\sum_{i\in \IndexSet{A}{\ell}} p^{\ell-i}\nu_i D[X]. 
\end{align*}
\end{Thm}
\begin{proof} 
We prove this by induction on $\ell$. If $\ell = 0$, then $\IndexSet{A}{0}=\{0\}$ 
and the assertion follows from Theorem~\ref{thm:NullI}.
Let $\ell\geq 1$. Then  $\IndexSet{A}{\ell}\setminus \{\ell\}\neq \emptyset$; let $k\leq \ell-1$  be the largest index in $\IndexSet{A}{\ell}\setminus \{\ell\}$. Then $\deg(\nu_{\ell})>\deg(\nu_k)$ and  $\deg(\nu_{\ell})=\deg(\nu_{k+1})$. Corollary~\ref{cor:degreeandppower} implies 
\begin{align*}
 \NullIovD{p^{\ell}} = \nu_{\ell}D[X] + p^{\ell-k} \NullIovD{p^k}.
\end{align*}
However, according to the induction hypothesis, 
\begin{align*}
 \NullIovD{p^k} =\sum_{i\in \IndexSet{A}{k}} p^{k-i}\nu_i D[X] 
\end{align*}
holds. In addition, it follows from Remark~\ref{rem:propofellindexset} that $\IndexSet{A}{\ell} = \IndexSet{A}{k}\cup \{\ell\}$ which completes the proof.
\end{proof}

\begin{Rem}\label{rem:nullidealgencase}
 For the general case, let $d = \prod_{i=1}^m p_i^{\ell_i}$ be the prime factorization of an element $d\in D$ and $c_i =  \prod_{j\neq i} p_j^{\ell_j}$. Let $\nu_{(p,\ell)}$ denote a $(p^{\ell})$-minimal polynomial and $\IndexSet{A}{(p,\ell)}$ the $\ell$-th index set of $A$ with respect to~the prime element $p$. 
 According to Theorem~\ref{thm:MinGenNullI} and Lemma~\ref{lem:CRT>primepowers}, the following holds:
\begin{align*}
 \NullIovD{d}(A) &= \sum_{i=1}^m %
                   \left( \sum_{j\in \IndexSet{A}{(p_i,\ell_i)}} c_i\,(p_i^{\ell_i-j}\nu_{(p_i,j)})D[X] \right) \\
                 &= \sum_{i=1}^m %
                   \left( \sum_{j\in \IndexSet{A}{(p_i,\ell_i)}} \left(\frac{d}{p_i^j}\,\nu_{(p_i,j)} \right)D[X] \right).
\end{align*}

\end{Rem}

The following assertions are technical observations which are useful later-on.

\begin{Cor}\label{cor:polyrepdegrees}
Let $D$ be a principal ideal domain and $p\in D$ a prime. Further, let $A\in \Matr{n}{D}$ be a square matrix over $D$ 
with $\ell$-th index set \IndexSet{A}{\ell} (for $\ell\geq 0$) and $\nu_i\in D[X]$ be $(p^i)$-minimal polynomials of $A$. If $f\in \NullIovD{p^{\ell}}(A)$, then 
\begin{align*}
  f \in \sum_{i\in \IndexSet{A}{\ell}^{[f]}} p^{\ell-i}\nu_i\,D[X]
\end{align*}
where $\IndexSet{A}{\ell}^{[f]} = \{\,i \in \IndexSet{A}{\ell} \mid \deg(\nu_i)\leq \deg(f)\,\}$.
\end{Cor}
\begin{proof}
We prove this by induction on $\ell$. Observe that, if $\deg(f) \geq \deg(\nu_{\ell})$, then $\IndexSet{A}{\ell}^{[f]}=\IndexSet{A}{\ell}$. In this case the assertion holds, according to Theorem~\ref{thm:MinGenNullI}. In particular, this is the case if $\ell = 0$ (which is the induction basis), since $\deg(f) \geq 0 = \deg(\nu_{0})$. 

Hence assume $\ell\geq 1$ and $\deg(f) < \deg(\nu_{\ell})$.  Then  $\ell\notin \IndexSet{A}{\ell}^{[f]}$, and, 
by Corollary~\ref{cor:degreeandppower}, $f=ph$ with $h \in \NullIovD{p^{\ell-1}}$.
According to the induction hypothesis, it follows that 
\begin{align*}
  h \in \sum_{i\in \IndexSet{A}{\ell-1}^{[h]}} p^{\ell-1-i}\nu_i\,D[X]. 
\end{align*}
Note that $\deg(f) = \deg(h)$ and therefore $\IndexSet{A}{\ell-1}^{[h]} = \IndexSet{A}{\ell-1}^{[f]}$.
We split into two cases, $\deg(\nu_{\ell}) > \deg(\nu_{\ell-1})$ and $\deg(\nu_{\ell}) = \deg(\nu_{\ell-1})$. 
According to Remark~\ref{rem:propofellindexset}, if $\deg(\nu_{\ell}) > \deg(\nu_{\ell-1})$, then $\IndexSet{A}{\ell-1} \cup \{\ell\} = \IndexSet{A}{\ell}$. Since $\ell \notin \IndexSet{A}{\ell}^{[f]}$ it follows that $\IndexSet{A}{\ell-1}^{[f]} = \IndexSet{A}{\ell}^{[f]}$.

If $\deg(\nu_{\ell}) = \deg(\nu_{\ell-1})$, then $\IndexSet{A}{\ell} = \{\ell\}\,\cup\, (\IndexSet{A}{\ell-1}\setminus\{\ell-1\})$, by Remark~\ref{rem:propofellindexset} again. 
However, $\ell\notin \IndexSet{A}{\ell}^{[f]}$ and $\ell-1\notin \IndexSet{A}{\ell-1}^{[f]}$ since $\deg(f) < \deg(\nu_{\ell}) = \deg(\nu_{\ell-1})$. Therefore $\IndexSet{A}{\ell-1}^{[f]} = \IndexSet{A}{\ell}^{[f]}$ in this case too.
Hence, in both cases, the following holds:
\begin{align*}
  f= ph \in \sum_{i\in \IndexSet{A}{\ell}^{[f]}} p^{\ell-i} \nu_i\,D[X]. 
\end{align*}
\end{proof}

For $i\geq 1$, let $\nu_i\in D[X]$ be $(p^i)$-minimal polynomials and $\mu_A\in D[X]$ the minimal polynomial of $A$. Then, by definition, 
\begin{align*}
 \pdeg = \deg(\nu_1) \leq   \cdots \leq \deg(\nu_{\ell-1}) \leq \deg(\nu_{\ell}) \leq \cdots \leq \deg(\mu_A) = \mathsf{d}_A.
\end{align*}
In particular, this sequence of degrees stabilizes. The following proposition states that there always exists an $m$ such that every $(p^m)$-minimal polynomial has degree $\mathsf{d}_A$, that is, the sequence stabilizes always at the value $\mathsf{d}_A$.
\begin{Pro}\label{pro:degstabatell}
 Let $D$ be a principal ideal domain and $p\in D$ a prime element. Further, let $A\in \Matrn{D}$ with 
 minimal polynomial $\mu_A\in D[X]$ and $\mathsf{d}_A = \deg(\mu_A)$. 
 If $\nu_i$ are $(p^i)$-minimal polynomials of $A$ for $i\geq 0$, then there exists $m\in \N$ such that for all $\ell\geq m$, $\deg(\nu_{\ell})=\mathsf{d}_A$ holds.
\end{Pro}
\begin{proof}
Since $\deg(\nu_i)\leq \deg(\mu_A)$ and $(\deg(\nu_i))_{i\geq 1}$ is a non-decreasing sequence in \N, there exists $m\in \N$ such that  $\deg(\nu_{m}) = \deg(\nu_{m+k})$ for all $k\geq 0$. We set $d= \deg(\nu_m)$ and show $d = \mathsf{d}_A$. Note that $d\leq \mathsf{d}_A$, and therefore it suffices  to show $d\geq \mathsf{d}_A$. 

Since $\nu_{m+k + 1} - \nu_{m+k}\in \NullIovD{p^{m+k}}$ is a polynomial with degree less than $\deg(\nu_m)$, it follows from Corollary~\ref{cor:degreeandppower} that
\begin{align*}
 \nu_{m+k + 1} - \nu_{m+k}\in p^{k+1}D[X].
\end{align*}
For $0\leq i \leq d$, let $a_{i}^{(k)}$ be the coefficient of $X^i$ of the polynomial $\nu_{m+k}$. 
Then $(a_{i}^{(k)})_{k\geq 0}$ are $p$-adic Cauchy sequences in $D$. 
Therefore $\nu = \lim_{k\rightarrow \infty} \nu_{m+k}$ is a polynomial over the $p$-adic completion $\cp{D}$ of $D$ 
with coefficients $a_i=\lim_{k\rightarrow \infty} a_i^{(k)}$  and $d=\deg(\nu)$.
Since,  $\nu_{m+k}$ is a monic polynomial for all $k$, it follows that $\nu$ is a monic polynomial too.

Further $\nu(A) = 0$, and hence $\nu\in \NullI[\cp{D}](A)$.
Now, let \cp{K} be the quotient field of \cp{D}. Then $\cp{K}$ is a field extension of $K$. Since the minimal polynomial is invariant under field extensions,  it follows that $\NullI[\cp{K}](A) = \mu_A \cp{K}[X]$. However, $\cp{D}$ is integrally closed in $\cp{K}$, and therefore  $\NullI[\cp{D}](A) = \mu_A \cp{D}[X]$.
Hence $\mu_A \,|\, \nu$ which implies in particular that $\mathsf{d}_A \leq \deg(\nu) = d$.
\end{proof}

We can conclude, that it suffices to determine a finite number of $(p^i)$-minimal polynomials in order to describe the ideals $\NullIovD{p^{\ell}}(A)$ for all $\ell\geq 0$.
\begin{Cor}\label{cor:mlargeenough}
Let $D$ be a principal ideal domain and $p\in D$ a prime element. Further, let $A\in \Matrn{D}$ 
and $\mu_A\in D[X]$ the minimal polynomial of $A$. 
Then there exists $m\in \N$ such that for all $k\geq 0$ the following holds:
\begin{align*}
\NullIovD{p^{m+k}}(A) = \mu_A D[X] + p^{k}\NullIovD{p^{m}}(A).
\end{align*}
\end{Cor}
\begin{proof}
For $i\ge 0$, let $\nu_i$ be a $(p^i)$-minimal polynomial of $A$. 
Then there exists an $m\in\N$ such that $\deg(\mu_A) = \deg(\nu_{m+1})$, according to Proposition~\ref{pro:degstabatell}. Hence, $\mu_A$ is a $(p^{m+k+1})$-minimal polynomial for all $k\geq 0$ and the assertion follows from Corollary~\ref{cor:degreeandppower} (with $j = m+1$).
\end{proof}

\subsection{Diagonal matrices}

Although we know that  $(p^{\ell})$-minimal polynomials exist, it is in general not clear how to determine them algorithmically. However, in the special case of diagonal matrices it is possible to compute them explicitly. Let $A = \diag(a_1,\ldots,a_n)$ be a diagonal matrix over $D$, $p\in D$ a prime element, $\ell\in \N$ and $f\in D[X]$ a polynomial. Then $f(A) = \diag(f(a_1),\ldots,f(a_n))$ holds and therefore 
\begin{align*}
\forall\, f\in D[X]: \; \left( f \in \NullIovD{p^{\ell}}(A) \quad \Longleftrightarrow \quad \forall\,i \in \{1\,\ldots, n\}:\; f(a_i) \in p^{\ell}D \right).
\end{align*}
However, the set of polynomials which maps the elements \textenum{a_1}{a_n} to multiples of $p^{\ell}$ can be determined using Bhargava's $p$-orderings, cf.~\cite{Bhargava97} and~\cite{Bhargava2000}. We explain his approach here in the special case of a principal ideal domain (although it is applicable in the more general case of a Dedekind domain by looking at prime ideals instead of  prime elements).
\begin{Def}
  Let $S$ be a non-empty subset $S$ of $D$. A \df{$p$-ordering of $S$} is a sequence $(b_k)_{k\geq 0}$ which is defined iteratively in the following way:
\begin{enumerate}
 \item Choose $b_0\in S$ arbitrary.
 \item If \textenum{b_0}{b_{k-1}} are already known, then choose $b_k\in S$ as an element such that $\val_p((b_k-b_{0})(b_k-b_{1})\cdots(b_k-b_{k-1}))$ is minimal, where $\val_p$ denotes the $p$-adic valuation on $D$.
\end{enumerate}
\end{Def}
In general, there is more than one $p$-ordering of a set $S$ (except $|S| = 1$) and
for each $p$-ordering $(b_k)_{k\geq 0}$ of $S$ we have the sequence  of $p$ powers $p^{\val_p((b_k-b_{0})(b_k-b_{1})\cdots(b_k-b_{k-1}))}$ (with the usual convention ``$p^{\infty} = 0$''). 
 Bhargava shows that the sequences of $p$ powers  of any two $p$-orderings are the same (cf.~\cite[Theorem~1]{Bhargava97}). Hence, these $p$ powers depend only on $S$ and not on the choice of the $p$-ordering. This motivates the following definition.
\begin{Def}
 Let $S$ be a non-empty subset $S$ of $D$ and $(b_k)_{k\geq 0}$ a $p$-ordering of $S$. For $k\geq 0$ let 
\begin{align*}
 v_k(S,p) = p^{\val_p((b_k-b_{0})(b_k-b_{1})\cdots(b_k-b_{k-1}))}D.
\end{align*}
Then $(v_k(S,p))_{k\geq 0}$ is called the \df{associated $p$-sequence of $S$}.
\end{Def}

Note that $v_0(S,p) = D$. By definition, $p$-orderings satisfy the following property
\begin{align}\label{eq:bhargava}
 \forall\, a\in S:\; p^{\val_p((a-b_{0})(a-b_{1})\cdots(a-b_{k-1}))} \in v_k(S,p). 
\end{align}
Therefore, the associated $p$-sequence of $S$ forms a descending chain of ideals, that is, $v_{k+1}(S,p) \subseteq v_k(S,p)$ for all $k\geq 0$. In particular, if $S$ is finite, then $v_{k}(S,p) = \Nl$ for $k\geq |S|+1$. Moreover, the property in~\eqref{eq:bhargava} implies that the polynomials of the form $f_k = (X-b_{0}) \cdots (X-b_{k-1})$ satisfy $f_k(S) \subseteq  v_k(S,p)$ for $k\geq 0$. In fact, the polynomials $f_k$ are indeed a suitable choice for our purpose. The following theorem allows us to deduce the desired properties.

\begin{Thm}{(\cite[Theorem~11]{Bhargava97})}\label{thm:bhargava}
 Let $S$ be a subset of a principal ideal domain $D$, and $f\in D[X]$ be a primitive polynomial of degree $k$. If $I_f$ denotes the smallest ideal of $D$ such that $f(S)\subseteq I_f$, then $v_k(S,p) \subseteq I_f$. Moreover, if $(b_j)_{j\geq 0}$ is a $p$-ordering of $S$, then the polynomial
 \begin{align*}
  g = (X-b_{0})(X-b_{1}) \cdots (X-b_{k-1}) 
 \end{align*}
 is a polynomials of degree $k$ such that $I_g = v_k(S,p)$.
\end{Thm}

We can use this theorem to compute $(p^{\ell})$-minimal polynomials for the diagonal matrix $A = \diag(a_1,\ldots,a_n)$ over principal ideal domains. Let $S = \{a_1,\ldots,a_n\}$ be the set of diagonal elements of $A$ and $\sigma$ a permutation of $\{1,\ldots,n\}$ such that $(a_{\sigma(i)})_{i=1}^n$ is a $p$-ordering of $S$.
We set $f_k = (X-a_{\sigma(0)})(X-a_{\sigma(1)}) \cdots (X-a_{\sigma(k-1)})$.

For $\ell\in \N$, let $k$ be minimal such that $v_k(S,p) \subseteq p^{\ell}D$. Then, by Theorem~\ref{thm:bhargava}, $f_k(S) \subseteq p^{\ell}D$ and we claim that $f_k$ is a $(p^{\ell})$-minimal polynomial. Assume that $f\in D[X]$ is a monic polynomial with degree less than $k$ and $f(S)\subseteq p^{\ell}D$. Again by Theorem~\ref{thm:bhargava}, this implies $v_{k-1}(S,p)\subseteq I_f\subseteq p^{\ell}D$ which contradicts the choice of $k$.

To compute the $(p^{\ell})$-minimal polynomial of $A$ we therefore only have to compute a $p$-ordering of the set of diagonal elements of $A$.
To demonstrate this approach, we conclude this section with an example of a $3$$\times$$3$-matrix over \Z.
\begin{Exa}\label{ex:genset}
Let $A\in\Matr{3}{\Z}$ be defined as follows:
\begin{align*}
 A = %
   \begin{pmatrix}
       4 &    0  & 0  \\
       0  &  16  & 0  \\
       0  &   0  & 32 \\
   \end{pmatrix}
\end{align*}
Then $A$ has three, pairwise different eigenvalues over \Q and hence 
\begin{align*}
\mu_A = (X-4)(X-16)(X-32)
\end{align*}
is the minimal polynomial of $A$ over $\Q$. Since $\mu_A\in \Z[X]$, it is the (in this case uniquely determined) minimal polynomial (or \Nl-minimal polynomial) of $A$ over \Z. 

Let $p\in \Z$ be a prime element. Recall that we denote the residue classes modulo a prime element $p$ by  \Resp{\,.\,}.
Then $\Resp{A}$ has three different eigenvalues in \FacRing{\Z}{p}{} for all prime elements in \Z except for the primes $2$, $3$ and $7$.
Therefore, 
\begin{align*}
\mu_{\Resp{A}} = (X-\Resp{4})(X-\Resp{16})(X-\Resp{32}) \in \FacRing{\Z}{p}{}[X]
\end{align*}
is the minimal polynomial of $\Resp{A}$ over \FacRing{\Z}{p}{} for all $p\in \Primes\setminus\{2,3,7\}$. This implies $\pdeg(A) = \deg(\mu_A)$ for all $p\in  \Primes \setminus \{2,3,7\}$.
Therefore $\mu_A$ is a $(p^{\ell})$-minimal polynomial of $A$ and $\{0,\ell\}$ the $\ell$-th index set of $A$ with respect to~the prime $p$
for all prime elements $p\neq 2,3,7$ and all $\ell\geq 1$. Hence, according to Theorem~\ref{thm:MinGenNullI}, 
\begin{align*}
 \NullIovD{p^{\ell}}(A) = \mu_A\Z[X] + p^{\ell}\Z[X]
\end{align*}
holds for all $p\in \Primes\setminus\{2,3,7\}$ and all $\ell\geq 1$.
The cases $p=3$ and $p=7$ are similar, therefore, we only handle $p=3$.
Observe that $4,32,16,16,\ldots$ is an example of a $3$-ordering of the set $\{4,16,32\}$ and $D,D,(3),\Nl,\Nl,\ldots$ is the associated $3$-sequence of this set.
Following Bhargava's approach (which we explained above this example), it follows that 
%
 $f_2 = (X-4)(X-32)$
%
is a $(3)$-minimal polynomial and 
 $\mu_A = f_3 = (X-4)(X-32)(X-16)$
is a $(3^{\ell})$-minimal polynomial $\ell \geq 2$. Moreover, $\{0,1\}$ is the first and $\{0,1,\ell\}$ is the $\ell$-th index set of $A$ for $\ell \geq 2$ (with respect to~$3$).
Theorem~\ref{thm:MinGenNullI} implies
\begin{align*}
 \NullIovD{3}(A) = (X-4)(X-32)\,\Z[X] + 3\,\Z[X] 
\end{align*}
and, for all $\ell\geq 2$,
\begin{align*}
\NullIovD{3^{\ell}}(A) = \mu_A\Z[X] +3^{\ell-1}(X-4)(X-32)\,\Z[X] + 3^{\ell}\,\Z[X].
\end{align*}

It remains to consider the case $p=2$. The sequence $4,16,32,32,\ldots$ is an example of a $2$-ordering of the set $\{4,16,32\}$ and $D,(4),(64),\Nl,\Nl,\ldots$ is the associated $2$-sequence of this set. We use Bhargava's approach again; the results are displayed in Table~\ref{tab:exa_p2}.


\renewcommand{\arraystretch}{1.1}
\begin{table}[h!]
\begin{center}
 \begin{tabular}{|c|c|c|}
\hline
  $\ell$ &  \IndexSet{A}{\ell} & $(2^{\ell})$-minimal polynomial \\
\hline
\hline
 1,2 &  $\{0,\ell\}$ & $X-4$ \\
\hline
 3,4,5,6 & $\{0,2,\ell\}$ & $(X-4)(X-16)$ \\
\hline
$\geq 7$ & $\{0,2,5,\ell\}$ & $\mu_A$ \\
\hline
 \end{tabular}
\caption{$(2^{\ell})$-minimal polynomials of $A$}
\label{tab:exa_p2}
\end{center}
\end{table}

Finally, it is worth mentioning that even if the degrees of $(p^{\ell})$- and $(p^{\ell+1})$-minimal polynomials coincide, a $(p^{\ell})$-minimal polynomials is in general \textbf{not} a $(p^{\ell+1})$-minimal polynomial (while the reverse implication holds).
This is easily verified, once one observes that $X^2$ is both, an $(8)$- and a $(16)$-minimal polynomial, but it is not a $(32)$-minimal polynomial of $A$.

\end{Exa}

\section{\texorpdfstring{Module structure of $\FacRing{D}{p}{\ell}[A]$}{Module structure}}
\label{sec:modulestructure}

\newcommand{\Nld}{\ensuremath{\NullI[<d]}\xspace}
\newcommand{\Rld}{\ensuremath{\Rl[X]^{<d}}\xspace}

Throughout this section we fix the prime power $p^{\ell}\in D$ and write $\Rl$ for the residue class ring $\FacRing{D}{p}{\ell}$.
Let $A\in \Matrn{\Rl}$ be a square matrix with null ideal 
\begin{align*}
\NullI = \NullI[\Rl](A) = \NullIovD[\Rl]{\Nl}(A) = \{\,f\in \Rl[X] \mid f(A) =0\,\}. 
\end{align*}

Further, let $A'\in \Matrn{D}$ be a preimage of $A$ under the projection modulo $p^{\ell}$, that is, $\Res{A'}{p^{\ell}} = A$ where \Res{\,.\,}{p^{\ell}} denotes the residue class modulo $p^{\ell}$ (as introduced in~Notation and Conventions~\ref{conv:pid}). Then, according to Theorem~\ref{thm:MinGenNullI}, 
\begin{align*}
 \NullI &= \{\,\Res{f}{p^{\ell}}\in \Rl[X] \mid f\in \NullIovD{p^{\ell}}(A')\,\} 
          = \sum_{i\in \IndexSet{A}{\ell}\setminus \{0\}} \Res{p}{p^{\ell}}^{\ell-i}\Res{\nu_i}{p^{\ell}}\Rl[X]
\end{align*}
where \IndexSet{A}{\ell} is the $\ell$-th index set of $A'$ and $\nu_i$ are $(p^i)$-minimal polynomials of $A'$ (for $i\in \IndexSet{A}{\ell}\setminus \{0\} $).

\begin{NotConv}\label{def:pjRlminpoly}
 Let  $f'\in D[X]$ be a monic polynomial. Recall that, for $1\leq j\leq \ell$, $f'$ is a $(p^j)$-minimal polynomial of $A'$ if and only if $f=\Res{f'}{p^{\ell}}$ is a $(\Res{p^j}{p^{\ell}})$-minimal polynomial of $A$, see Remark~\ref{rem:defminpoly_ovD}. 

 For a better readability, we often write $p$ for the residue class \Res{p}{p^{\ell}} of $p$ modulo $p^{\ell}$ and say that $f\in \Rl[X]$ is a $(p^j)$-minimal polynomial of $A$ if it is a $(\Res{p^j}{p^{\ell}})$-minimal polynomial of $A$.
\end{NotConv}

 Note that the $\ell$-th index set of a matrix $A'\in \Matrn{D}$ only depends on the residue class of $A'$ modulo $p^{\ell}$, that is, if $A''\in \Matrn{D}$ is a matrix with $\Res{A'}{p^{\ell}} = \Res{A''}{p^{\ell}}$ (and therefore $\Res{A'}{p^{j}} = \Res{A''}{p^{j}}$ for all $1\leq j\leq \ell$), then $A'$ and $A''$ have equal $\ell$-th index sets, cf.~Remark~\ref{rem:index_set_resclass}.
\begin{Def}\label{def:reducedindexset_successor}
 Let $A\in \Matrn{\Rl}$ and $A'\in \Matrn{D}$ such that $A=\Res{A'}{p^{\ell}}$. If  $\IndexSet{A}{\ell}$ is the $\ell$-th index set of $A'$, then we call $\SIndexSet{A}{\ell} = \IndexSet{A}{\ell}\setminus \{0,\ell\}$ the \df{reduced index set of} $A$. 
 Further, for $i\in \IndexSet{A}{\ell} \setminus \{\ell\}$, we call $\nf{i}=\min\{i'\in \IndexSet{A}{\ell}\mid i'>i\}$  the \df{successor of $i$ in \IndexSet{A}{\ell}}.
\end{Def}
\begin{Rem}\label{rem:reduced_index_set}
 Let $A\in \Matrn{\Rl}$ with reduced index set \SIndexSet{A}{\ell}, and let $\nu_i\in \Rl[X]$ be $(p^j)$-minimal polynomials of $A$ (for $1\leq i \leq \ell$). Then $i\in \SIndexSet{A}{\ell}$ if and only if $\deg(\nu_i) < \deg(\nu_{i+1})$, cf.~Definition~\ref{def:indexset}.
 Further, note that if $i\in \SIndexSet{A}{\ell}$, then  $\deg(\nu_{\nf{i}}) = \deg(\nu_{i+1})$.
\end{Rem}

In this section we analyze the structure of the \Rl-module $\Rl[A]$. Since the null ideal of $A$ contains a monic polynomial, there exists a power of $A$ which can be written as an $\Rl$-linear combination of smaller powers of $A$. Therefore the module $\Rl[A]$ is finitely generated.
As a finitely generated module over a principal ideal ring, $\Rl[A]$ decomposes into cyclic \Rl-submodules, according to~\cite[Theorem~15.33]{Brown1993}.
We  compute such a decomposition exploiting its relation to the generating set of the null ideal $\NullI$ of $A$ which we determined in Theorem~\ref{thm:MinGenNullI} of the last section.
In particular, it turns out that the invariant factors of $\Rl[A]$ correspond to the elements in the reduced index set $\SIndexSet{A}{\ell}$ of $A$. Further, their multiplicities relate to the degrees of the $(p^j)$-minimal polynomials, see Remark~\ref{rem:connection_to_reduced_index_set}.
As the invariant factors are uniquely determined, this corroborates  the usefulness of the set of generators of the null ideal of $A$ which we determined in Section~\ref{sec:GeneratorsNullIdeal}.
To be more specific, Theorem~\ref{thm:decompositionRlA} below states that, if $\SIndexSet{A}{\ell}$ is the reduced index set of $A$ and $s_j= \deg(\nu_{\nf{j}})-\deg(\nu_j)$ for $j\in \SIndexSet{A}{\ell}$, then 
\begin{align}\label{eq:decomposition}
  \Rl[A]\simeq \Rl^{\pdeg} \oplus \bigoplus_{j\in\SIndexSet{A}{\ell}} (R_{\ell-j})^{s_j}
\end{align}
where $\pdeg = \deg(\nu_1)$ is the degree of the minimal polynomial of $A$ modulo $p$.
Roughly speaking,  the \Rl-free part $\Rl^{\pdeg}$ of the decomposition in~\eqref{eq:decomposition} indicates what happens in terms of classical linear algebra over the field $R_1$ while the torsion-part of $\Rl[A]$ relates to the set \SIndexSet{A}{\ell}.

In order to understand this connection, let $d$ be the degree of a $(p^\ell)$-minimal polynomial $\nu_{\ell}$. Then $A^d$ is an \Rl-linear 
combination of $I$, $A$, ..., $A^{d-1}$, and thus $\Rl[A] = \ModuleVar{\Rl}{I, A, \ldots, A^{d-1}}$. Hence the following sequence of \Rl-modules is exact.
\begin{align}
\begin{aligned}\label{exseq:RA}
       \Nl \longrightarrow  \ker(\psi)   \longrightarrow \Rl^{d}   \;      &\stackrel{\psi}{\longrightarrow} \;\Rl[A]  \longrightarrow  \Nl \\
                                                \basis{e}{i}\;    &\longmapsto\;                    \:A^{i-1}                     
\end{aligned}                                               
\end{align}
where \textenum{\basis{e}{1}}{\basis{e}{d}} is an arbitrary basis of $\Rl^{d}$.
It follows that 
\begin{align*}
 \Rl[A] \simeq \nicefrac{\Rl^d}{ \ker(\psi)}.
\end{align*}
Elements of $ \ker(\psi)$ correspond to relations between the matrices \textenum{I,A}{A^{d-1}} and therefore to  polynomials in the null ideal $\NullI$ of $A$ of degree less than $d$. Hence 
\begin{align}\label{eq:M_Nl}
\sum_{i=1}^d\lambda_i\basis{e}{i}\in  \ker(\psi) \quad \Longleftrightarrow \quad \sum_{i=1}^d\lambda_iX^{i-1} \in \NullI
\end{align}
where $\lambda_1,\ldots,\lambda_d \in \Rl$.
We exploit this equivalence and use a generating set of the null ideal \NullI of $A$ to compute a generating set of the module $\ker(\psi)$. Nevertheless, we need to be careful, since (as an ideal of $\Rl[X]$) \NullI is an $\Rl[X]$-module and $ \ker(\psi)$ is only an $\Rl$-module. Hence multiplication by $X$ needs to be dealt with when transferring a generating set of \NullI to a generating set of $ \ker(\psi)$.
For this purpose, set $\Rld = \{\,f\in \Rl[X] \mid \deg(f) < d\,\}$. Then
\begin{align}
 \begin{split}\label{eq:isophi}
  \varphi: \Rld \,          &\stackrel{\sim}{\longrightarrow} \, \Rl^d\\
          X^{i-1}     \quad &\longmapsto\;\,                        \basis{e}{i}
 \end{split}
\end{align}
is an \Rl-module isomorphism. Let 
\begin{align*}
\Nld &=\{\,f\in\NullI \mid \deg(f) < d\,\}
\end{align*}
be the set of all elements in \NullI{} of degree less than $d$. Then \Nld is an \Rl-module, and for $f_1,\ldots,f_r\in \Rld$, the following holds
\begin{align*}
  \Nld = \ModuleVar{\Rl}{f_1,\ldots,f_r} \quad \Longleftrightarrow \quad  \ker(\psi) = \ModuleVar{\Rl}{\varphi(f_1),\ldots,\varphi(f_r)}
\end{align*}
according to the equivalence in~\eqref{eq:M_Nl}. We modify the sequence in~\eqref{exseq:RA} accordingly to get the following exact sequence of $\Rl$-modules. 
\begin{align}
\begin{aligned}\label{exseq:RA2}
       \Nl \longrightarrow  \Nld   \longrightarrow \Rld  &\longrightarrow \Rl[A]  \longrightarrow  \Nl \\
                                         X^i\quad    &\longmapsto\;\;       A^{i}    
\end{aligned}
\end{align}

The following lemma describes which $\Rl[X]$-generating sets of $\NullI{}$ can be transferred to $\Rl$-generating sets of $\Nld$.

\begin{Lem}\label{lem:findRlgens}
Let $A\in \Matrn{\Rl}$ be a square matrix over \Rl and $d$ the degree of a $(p^{\ell})$-minimal polynomial of $A$. 
Further, let $f_1,\ldots, f_m$ be a generating set of the null ideal $\NullI{}$ of $A$ in $\Rl[X]$ such that 
\begin{enumerate}
 \item $\deg(f_1) < \cdots < \deg(f_m) = d $,
 \item $f_i = \Res{p^{t_i}}{p^{\ell}}\,g_i$ for  monic polynomials $g_i \in \Rl[X]$  ($1 \leq i\leq m$) and natural numbers $t_1 > \cdots > t_m$,
 \item $f \in \sum_{i\in \mathcal{I}^{[f]}}f_i\,\Rl[X]$ for all $f\in \NullI{}$, where $\mathcal{I}^{[f]} = \{\,1 \leq i \leq m \mid \deg(f_i) \leq \deg(f) \,\}$.
\end{enumerate}
Then
\begin{align*}
 \Nld = \sum_{i=1}^{m-1} \sum_{t=1}^{s_i}{(X^{t-1}f_i)} \,{\Rl}
\end{align*}
where $s_i = \deg(f_{i+1}) - \deg(f_i)$. 
\end{Lem}
\begin{proof}
The conditions on the degrees of the polynomials $f_i$ guarantee that $\deg(X^{t-1}f_i)<d$ for $1\leq i\leq m-1$ and $1\leq t \leq s_i$. Hence the inclusion ``$\supseteq$'' is easily seen and it suffices to show ``$\subseteq$''. Let $f\in \Nld$. We prove this by induction on $\deg(f)$. 

For the basis, let $0\neq f\in \Nld$ be a polynomial of minimal degree in \Nld, that is, $\deg(f)\leq\deg(g)$ for all $g\in \Nld$. Since 
\begin{align*}
f \in \sum_{i\in \mathcal{I}^{[f]}}f_i\,\Rl[X] 
\end{align*}
it follows that $\mathcal{I}^{[f]} = \{\,1\leq i \leq m \mid \deg(f_i)\leq \deg(f)\,\}\neq \emptyset$. Therefore $\deg(f) = \deg(f_1)$ and $\mathcal{I}^{[f]} = \{1\}$ (since $\deg(f_j) > \deg(f_1)$ for $j>1$).
Hence $f= rf_1$ for $r\in \Rl$ which proves the basis. 

Assume now $f\in \Nld$ with $\deg(f)>\deg(f_1)$. Let $1\leq k < m$ such that $\deg(f_k) \leq \deg(f) < \deg(f_{k+1})$.
Then, $f\in \sum_{i=1}^kf_i\Rl[X] \subseteq p^{t_k}\Rl[X]$ according to our assumptions on the polynomials $f_i$ (where we write $p$ for its residue class \Res{p}{p^{\ell}}).
Let $f'\in \Rl[X]$ (with $\deg(f)=\deg(f')$) such that $f = p^{t_k}f'$.
Since $f_k=p^{t_k}g_k$ for a monic polynomial $g_k\in \Rl[X]$, 
 there exist $q,r\in \Rl[X]$ with $\deg(r)<\deg(g_k) =\deg(f_k)$ such that
\begin{align}\label{eq:defhk}
 f' = qg_k + r.
\end{align}
Therefore 
\begin{align*}
 f = qf_k + p^{t_k}r
\end{align*}
which implies $p^{t_k}r\in \Nld$, and we can apply the induction hypothesis to $p^{t_k}r$. Hence 
\begin{align*}
p^{t_k}r\in \sum_{i=1}^{m-1} \sum_{t=1}^{s_i}{(X^{t-1}f_i)} \,{\Rl}.
\end{align*}
Since $\deg(f') =\deg(f) <\deg(f_{k+1})$, Equation~\eqref{eq:defhk} implies $\deg(q) = \deg(f)-\deg(f_k)<\deg(f_{k+1})-\deg(f_k) = s_k$. Therefore
\begin{align*}
 qf_k \in \sum_{t=1}^{s_k} (X^{t-1}f_k)\Rl
\end{align*}
and the assertion follows for $f = qf_k + p^{t_k}r$.
\end{proof}
According to Corollary~\ref{cor:polyrepdegrees}, any generating set of the form $\{\,p^{\ell-i}\nu_i\mid i\in \SIndexSet{A}{\ell}\,\}$, where $\nu_i\in \Rl[X]$ are $(p^i)$-minimal polynomials, satisfies the conditions  of Lemma~\ref{lem:findRlgens}. This allows us to prove the following theorem which is the main result of this section.

\begin{Thm}\label{thm:decompositionRlA}
Let $A\in \Matrn{\Rl}$  and $\nu_i\in \Rl[X]$ be $(p^i)$-minimal polynomials with $d_i =\deg(\nu_i)$ for $0\leq i \leq \ell$. Then  
\begin{align*}
\Rl[A] \simeq \bigoplus_{i=0}^{\ell-1}(R_{\ell - i})^{d_{i+1}-d_i}.
\end{align*}
Further, let $\SIndexSet{A}{\ell}$ be the reduced  index set of $A$ and 
$s_i = \deg(\nu_{\nf{i}}) - \deg(\nu_{i})$ for $i\in \SIndexSet{A}{\ell}$, then
\begin{align*}
\Rl[A] \simeq \Rl^{\pdeg} \oplus \bigoplus_{i\in \SIndexSet{A}{\ell}}(R_{\ell - i})^{s_i}
\end{align*}
where $\pdeg=\deg(\nu_1)$ is the $p$-degree of $A$. 
\end{Thm}
\begin{proof}
First, we show that the two decompositions of $\Rl[A]$ given in the theorem, are isomorphic. 
Recall that $\nu_0 = 1$ and $d_0 = 0$. Hence $\Rl^{\pdeg} = R_{\ell-i}^{d_{i+1}-d_i} $ for $i=0$. 
Let now $i\geq 1$. By Remark~\ref{rem:reduced_index_set}, an element $1 \leq i < \ell$ is in the reduced index set \SIndexSet{A}{\ell} of $A$ if and only if $d_i < d_{i+1}$, and if one of these equivalent conditions is satisfied, then $d_{i+1} = d_{\nf{i}}$.
Therefore, $i\in \SIndexSet{A}{\ell}$ if and only if $R_{\ell-i}^{d_{i+1}-d_i} \neq \Nl$ and then 
$(R_{\ell - i})^{s_i}=(R_{\ell - i})^{d_{i+1}-d_i}$.
Hence the two representations are isomorphic and it suffices to show that 
\begin{align*}
\Rl[A] \simeq \Rl^{\pdeg} \oplus \bigoplus_{i\in \SIndexSet{A}{\ell}}(R_{\ell - i})^{s_i}.
\end{align*}
According to Corollary~\ref{cor:polyrepdegrees} the polynomials in $\{\,p^{\ell-i}\nu_i\mid i\in \SIndexSet{A}{\ell}\,\}$ satisfy the conditions of  Lemma~\ref{lem:findRlgens}, and therefore 
\begin{align*}
  \Nld = \sum_{i\in \SIndexSet{A}{\ell}}\sum_{t=1}^{s_i} (p^{\ell-i}X^{t-1}\nu_i)\,\Rl.
\end{align*}

Since $s_i =\deg(\nu_{\nf{i}}) - \deg(\nu_{i}) $, it follows that 
\begin{align*}
\delta: \{\,(i,t)\mid i\in \SIndexSet{A}{\ell}, 1\leq t \leq s_i\,\} \; &\stackrel{\sim}{\longrightarrow} \; \{\,\pdeg+1,\ldots,d\,\}\\
                       (i,t)            \quad\quad  &\longmapsto  \quad                    \deg(\nu_{i})+t 
\end{align*}
is a bijection. 
For $1\leq j\leq d$, we define
\begin{align*}
 \basis{b}{j} = \begin{cases}
                        X^{j-1} & \text{ if } 1\leq j \leq \pdeg \\
                        X^{t-1}\nu_{i} & \text{ if } \pdeg +1\leq j = \delta(i,t)\leq d\ .
                       \end{cases}
\end{align*}
Observe that $\deg(\basis{b}{j}) = j-1$. Hence $\basis{b}{1},\ldots,\basis{b}{d}$ is a basis of $\Rld$. 
Together with the exact sequence~\eqref{exseq:RA2}, this implies 
\begin{align*}
 \Rl[A] &\simeq \nicefrac{\Rld}{ \Nld} \\
        &\simeq \bigoplus_{i=1}^{\pdeg} \basis{b}{i}\,\Rl
                \oplus 
                \bigoplus_{i\in \SIndexSet{A}{\ell}} 
                \bigoplus_{t=1}^{s_i} 
                \nicefrac{%
                  \basis{b}{\delta(i,t)}\Rl%
                }{%
                (p^{\ell-i}\, \basis{b}{\delta(i,t)})\Rl %
                } \\
        &\simeq \Rl^{\pdeg} \oplus \bigoplus_{i\in \SIndexSet{A}{\ell}} (R_{\ell - i})^{s_i} \ .
\end{align*}

\end{proof}

\begin{Rem}\label{rem:connection_to_reduced_index_set}
Let the notation be as in Theorem~\ref{thm:decompositionRlA}. If $\SIndexSet{A}{\ell} = \{i_1,\ldots, i_r\}$ with $i_1 < \cdots <i_r< i_{r+1} =\ell$. Then $s_{i_j} = \deg(\nu_{i_{j+1}})-\deg(\nu_{i_{j}})$ for $1\leq j\leq r$.
According to Theorem~\ref{thm:decompositionRlA}, the uniquely determined invariant factors of $\Rl[A]$ (with multiplicities) are 
\begin{align*}
 \underbrace{1,\ldots,1}_{\pdeg}, \underbrace{p^{\ell-i_1},\ldots,p^{\ell-i_1}}_{s_{i_1}},\ldots,\underbrace{p^{\ell-i_{r}},\ldots,p^{\ell-i_{r}}}_{s_{i_{r}}}.
\end{align*}
Note that the occurring exponents $\ell-i_1,\ldots,\ell-i_{r}$ of the invariant factors correspond to the elements of the set $\SIndexSet{A}{\ell}$.
Further, if $\nu_k\in \Rl[X]$ is a $(p^k)$-minimal polynomial of $A$ (for $1\leq k\leq \ell$), then there exists $1\leq u\leq r+1$ such that $\deg(\nu_k) = \deg(\nu_{i_u})$ and 
\begin{align*}
  \deg(\nu_k) = \sum_{i=0}^{k-1}(d_{i+1}-d_i) =\pdeg + \sum_{j=1}^{u-1}s_{i_j}.
\end{align*}
\end{Rem}

Recall that the $\ell$-th index set of a matrix defines a generating set of the null ideal $\NullI[\Rl](A)$ of $A$ consisting of polynomials of the form $p^{\ell-j}\nu_j$. Per definition, \SIndexSet{A}{\ell} depends on the degrees of these polynomials. In particular, observe that  $\SIndexSet{A}{\ell} = \emptyset$ if and only if  $\deg(\nu_{\ell}) = \deg(\nu_{1}) = \pdeg$. Together with Theorems~\ref{thm:MinGenNullI} and \ref{thm:decompositionRlA} this implies the following corollary.
\begin{Cor}
 Let $A\in \Matrn{\Rl}$ with $\ell$-th index set $\IndexSet{A}{\ell}$, $(p^{\ell})$-minimal polynomial $\nu_{\ell}$   and $p$-degree \pdeg. Then the following assertions are equivalent:
 \begin{enumerate}
  \item $\Rl[A]\simeq \Rl^{\pdeg}$
  \item $\deg(\nu_{\ell}) = \pdeg$
  \item $\NullI[\Rl](A) = \nu_{\ell}\Rl[X]$
 \end{enumerate}
\end{Cor}

We can reformulate this in terms of matrices with entries in $D$.
\begin{Cor}\label{cor:degnuell=dp}
Let $A\in \Matrn{D}$ and $\ell\in \N$. Further, let $\nu_j\in D[X]$  be $(p^j)$-minimal polynomials of $A$ for $1\leq j\leq \ell$ and \Res{A}{p^j} be the image of $A$ under projection modulo $p^j$. The following assertions are equivalent.
\begin{enumerate}
 \item  $\NullIovD[D]{p^{\ell}}(A)=\nu_{\ell}D[X] + p^{\ell}D[X]$.
 \item  $\NullIovD[D]{p^{j}}(A)= \nu_j D[X]+ p^jD[X]$  for all  $1\leq j \leq \ell$.
 \item  $R_j[\Res{A}{p^j}] \simeq R_j^{\pdeg}$ for all  $1\leq j \leq \ell$.
 \item  $\deg(\nu_{\ell}) = \pdeg$.
 \item  $\nu_{\ell}$ is a $(p^j)$-minimal polynomial of $A$ for all $1\leq j \leq \ell$.
\end{enumerate}
\end{Cor}

Recall, that Proposition~\ref{pro:degstabatell} states, that for $A\in \Matrn{D}$, there exists $m\in \N$ such that $\deg(\nu_{m+k}) = \deg(\nu_A)$     for all $k\geq 0$. Then $\SIndexSet{A}{m+k} = \SIndexSet{A}{m}$, cf.~Remark~\ref{rem:propofellindexset}. Together with Theorem~\ref{thm:decompositionRlA} we conclude this section with a final corollary.
\begin{Cor}
 Let $A\in \Matrn{D}$ and $\nu_j$ be $(p^j)$-minimal polynomials for $j\geq 1$. Further, let \Res{A}{p^j} be the image of $A$ under  projection modulo $p^j$.
Then there exists $m\in \N$ such that for all $\ell \geq m$ the following holds
        \begin{align*}
           R_{\ell}[\Res{A}{p^{\ell}}] \simeq R_{\ell}^{\pdeg} \oplus \bigoplus_{j\in \SIndexSet{A}{m}}(R_{\ell-j})^{s_j}
        \end{align*}
where $\SIndexSet{A}{m}$ is the reduced index set of $\Res{A}{p^{m}}$ and $s_j = \deg(\nu_{\nf{j}})-\deg(\nu_j)$ for $j\in \SIndexSet{A}{m}$.
In particular, $R_{\ell}[\Res{A}{p^{\ell}}]$ decomposes into $\deg(\mu_A)$ non-zero cyclic summands.
\end{Cor}

\section{Integer-valued polynomials on one matrix}
\label{sec:IntValOnA}

This section is dedicated to the application of the results of  Section~\ref{sec:GeneratorsNullIdeal} in the context of integer-valued polynomials on a single matrix. 
Again, let $D$ be a principal ideal domain with quotient field $K$ and $A\in \Matrn{D}$ be a square matrix with entries in $D$. We want to determine the ring \IntA of all integer-valued polynomials on $A$, that is,  
\begin{align*}
 \IntA =  \{\,f \in K[X] \mid f(A) \in \Matrn{D}\,\}.
\end{align*}
Once we have an explicit description of \IntA, we can determine the ring of images of $A$ under \IntA, that is, 
\begin{align*}
 \IntImA =  \{\,f(A) \mid f\in \IntA\,\}.
\end{align*}
For the ring of integer-valued polynomials on a single matrix $A$, the following inclusion holds
\begin{align*}
  \mu_AK[X] + D[X] \subseteq \IntA.
\end{align*}
There are both instances in which equality holds, and instances in which the inclusion is strict. 
If equality holds, it is readily seen that $\IntImA = D[A]$, that is, all images of $A$ under integer-valued polynomials on $A$ can be written as $g(A)$ with $g\in D[X]$. As far as the images of $A$ are concerned, the integer-valued polynomials in $K[X]\setminus D[X]$ do not contribute anything new in this case. In fact, as the next proposition states, the reverse implication holds too. (Thanks to Giulio Peruginelli for pointing this out.)
\begin{Pro}\label{prop:image_equiv}
Let $D$ be a principal ideal domain and $A\in \Matrn{D}$ with minimal polynomial $\mu_A\in D[X]$. Then the following assertions are equivalent: 
\begin{enumerate}
 \item $\IntA =\mu_AK[X] + D[X]$
 \item $\forall\, f \in \IntA\setminus D[X]:\; \deg(f) \geq \deg(\mu_A)$
 \item $\IntImA = D[A]$
\end{enumerate}
\end{Pro}
\begin{proof}
For the implication from \textit{1.}~to \textit{2.}~let $f\in \IntA\setminus D[X]$, then there exist $h\in K[X]$ and $g\in D[X]$ such that $f = h\mu_A + g$. Since $\mu_A\in D[X]$, we can assume that $\deg(g)<\deg(\mu_A)$. Further, $f\notin D[X]$ implies that $f\neq g$ and $h\neq 0$. Therefore $\deg(f) =\deg(h)+\deg(\mu_A) \geq \deg(\mu_A)$.

For the implication \textit{2.}~to \textit{3.}~let $f\in \IntA$. By polynomial division, there exists $q,r\in K[X]$ such that $f = q\mu_A + r$ and  $\deg(r)< \deg(\mu_A)$. The assumption in \textit{2}.~implies that $r\in D[X]$ and therefore $f(A) = r(A) \in D[A]$.
 
And finally we show that \textit{3}.~implies \textit{1}. Again, let $f\in \IntA$. Then, since $\IntImA=D[A]$ holds by assumption, there exists $g\in D[X]$ such that $f(A) = g(A)$. This further implies that $f-g \in \NullI[K](A) = \mu_AK[X]$ and hence there exists $h\in K[X]$ such that $f-g = h\mu_A$. The assertion follows.
\end{proof}
\begin{Rem}
The result above holds more generally over arbitrary domains $D$ under the additional assumptions that the minimal polynomial $\mu_A$ is an element of $D[X]$. Moreover, it is worth mentioning this assumption is only needed in the proof of the implication from \textit{1.}~to \textit{2}. 
\end{Rem}

However, in general, $\deg(\mu_A)$ is not a lower bound for the degree of polynomials in $\IntA\setminus D[X]$.
Let $f=\frac{g}{d} \in K[X]$ with $g\in D[X]$ and $d\in D$ and $d = \prod_{i=1}^m p_i^{\ell_i}$  the prime factorization of $d$.
Then the following assertions are equivalent:
\begin{enumerate}
 \item $f \in \IntA $
 \item $g(A) \equiv 0 \mod d\Matrn{D}$
 \item $g(A) \equiv 0 \mod p_i^{\ell_i}\Matrn{D}$ for all $1\leq i \leq m$
\end{enumerate}

The results of Section~\ref{sec:GeneratorsNullIdeal} provide the tools to give an explicit description of the ring \IntA of integer-valued polynomials on $A$. 

\begin{Thm}\label{thm:intvalA_desc}
Let $D$ be a principal ideal domain and $A\in \Matrn{D}$ with minimal polynomial $\mu_A\in D[X]$. 
Then there exists a finite set $\mathcal{P}_A\subset \Primes$ of prime elements of $D$ and natural numbers $m_p\in \N$ for $p\in \mathcal{P}_A$ such that 
\begin{align*}
 \IntA = \mu_AK[X] + D[X] + \sum_{p\in\mathcal{P}_A} \sum_{j\in \IndexSet{A}{(p,m_p)}} \frac{\nu_{(p,j)}}{p^j}D[X]
\end{align*}
where $\nu_{(p,j)}\in D[X]$ are $(p^j)$-minimal polynomials of $A$ for $j\geq 0$, and \IndexSet{A}{(p,m_p)} is the $m_p$-th index set of $A$ with respect to~the prime $p$. 
\end{Thm}
\begin{proof}
It suffices to show ``$\subseteq$''.
Recall that $\NullIovD{d}(A) = \NullIovD[D]{d}(A) = \{\,f\in D[X]\mid f(A)\in d\Matrn{D}\,\}$ and that $\NullIovD{0}(A) = \NullI(A) = \mu_AD[X] \subseteq D[X] = \NullIovD{1}(A)$ and hence
\begin{align*}
 \IntA = \sum_{d\in D\setminus \{0\}}\frac{1}{d}\, \NullIovD{d}(A).
\end{align*}
According to Lemma~\ref{lem:CRT>primepowers}, this implies 
\begin{align}\label{eq:easy_rep_inta}
 \IntA = \sum_{p\in \Primes}\sum_{\ell\in\N} \frac{1}{p^{\ell}}\, \NullIovD{p^{\ell}}(A).
\end{align}
First, we show that there exists a finite subset $\mathcal{P}_A\subseteq \Primes$ such that the following holds
\begin{align}\label{eq:nulliovd_forall_but_fin}
 \forall\, p\in\Primes\setminus \mathcal{P}_A:\; \NullIovD{p^{\ell}}(A)  = \mu_AD[X] +p^{\ell}D[X].
\end{align}
Considered as a matrix over $K$, $A$ is similar to its rational canonical form $C$, cf.~\cite{Roman08}. 
Let $\mu_1\,|\,\cdots\,|\,\mu_r = \mu_A$ be the invariant factors of $A$. Then there exists a matrix $T\in \GLn{K}$ such that
\begin{align*}
 T^{-1}AT = C = \Comp{\mu_A} \oplus \cdots \oplus \Comp{\mu_1} 
\end{align*}
where $\Comp{f}$ denotes the companion matrix of a monic polynomial $f$. Since $D$ is a principal ideal domain, it is integrally closed. As mentioned above, this implies $\mu_A\in D[X]$. Indeed, this implies that $\mu_i\in D[X]$ for all $1\leq i\leq r$, since they are all monic divisors of the characteristic polynomial $\chi_A\in D[X]$, cf.~\cite[Ch.~5,~§1.3,~Prop.~11]{Bourbaki_CommAlg}. 
Therefore the rational canonical form $C$ of $A$ is a matrix with entries in $D$. 

However, in general, $A$ is not similar to $C$ over the domain $D$, that is, we cannot assume $T\in \GLn{D}$.
Let $\mathcal{P}_A \subseteq \Primes$ be the set of prime elements which occur as divisors of 
the denominators of the entries of $T$ or its inverse $T^{-1}$. Then $\mathcal{P}_A$ is finite and  $T,T^{-1}$ are invertible matrices over the localization \Loc{D}{p} of $D$ at $p$ for all $p\in \Primes\setminus\mathcal{P}_A$ and we can reduce the equation above  modulo all $ p\in \Primes\setminus\mathcal{P}_A$: 
\begin{align*}
 \Resp{T}^{-1}\Resp{A}\Resp{T} = \Resp{T^{-1}AT} = \Resp{C} = \Comp{\Resp{\mu_A}} \oplus \cdots \oplus \Comp{\Resp{\mu_1}} 
\end{align*}
(where we identify the residue fields of $D$ and \Loc{D}{p} modulo $p$).
It is well known, that a monic polynomial $f$ is the minimal polynomial of its companion matrix \Comp{f} over any domain. 
Therefore $\Resp{\mu_A}$ is the minimal polynomial of \Comp{\Resp{\mu_A}}. Further,  $\Resp {\mu_A}(\Comp{\Resp{\mu_i}}) = 0$ holds since $\mu_i\,|\,\mu_A$ for all $1\leq i \leq m$.
Hence  $\mu_A$ is a $(p)$-minimal polynomial for all $p\in \Primes\setminus\mathcal{P}_A$, which implies the assertion in~\eqref{eq:nulliovd_forall_but_fin} above, according to Corollary~\ref{cor:degnuell=dp}.

Thus, Equations~\eqref{eq:easy_rep_inta} and \eqref{eq:nulliovd_forall_but_fin} imply
\begin{align}\label{eq:intermediate_rep_inta}
 \IntA = \mu_AK[X] + D[X] + \sum_{p\in\mathcal{P}_A} \sum_{\ell\geq 1} \frac{1}{p^{\ell}}\, \NullIovD{p^{\ell}}(A).
\end{align}
Further, by Corollary~\ref{cor:mlargeenough}, for all prime elements $p\in \mathcal{P}_A$, there exists $m_p\in \N$ such that for all $\ell\geq m_p$
\begin{align*}
  \NullIovD{p^{\ell}}(A) = \mu_AD[X] + p^{\ell-m_p}\NullIovD{p^{m_p}}(A)
\end{align*}
holds, and we can restrict the inner sum in Equation~\eqref{eq:intermediate_rep_inta} to all $1\leq \ell\leq m_p$.
And finally, since $p\NullIovD{p^{\ell-1}}(A) \subseteq \NullIovD{p^{\ell}}(A)$, it follows hat $\frac{1}{p^{\ell-1}}\NullIovD{p^{\ell-1}}(A) \subseteq \frac{1}{p^{\ell}}\NullIovD{p^{\ell}}(A)$. Hence
\begin{align*}
 \sum_{\ell=1}^{m_p} \frac{1}{p^{\ell}}\NullIovD{p^{\ell}}(A) = \frac{1}{p^{m_p}}\NullIovD{p^{m_p}}(A).
\end{align*}
Then, Theorem~\ref{thm:MinGenNullI} implies 
\begin{align*}
 \IntA = \mu_AK[X] + D[X] + \sum_{p\in\mathcal{P}_A} \sum_{j\in \IndexSet{A}{(p,m_p)}} \frac{\nu_{(p,j)}}{p^j}D[X].
\end{align*}
\end{proof}

\begin{Cor}
Let $D$ be a principal ideal domain and $A\in \Matrn{D}$ with minimal polynomial $\mu_A\in D[X]$.  
Then there exists a finite set $\mathcal{P}_A\subset \Primes$ and natural numbers  $m_p\in \N$ for $p\in \mathcal{P}_A$ such that 
\begin{align*}
 \IntImA = D[A] + \sum_{p\in\mathcal{P}_A} \sum_{j\in \IndexSet{A}{(p,m_p)}} \frac{\nu_{(p,j)}(A)}{p^j}D[A]
\end{align*}
where $\nu_{(p,j)}\in D[X]$ are $(p^j)$-minimal polynomial of $A$ for $j\geq 0$, and \IndexSet{A}{(p,m_p)} is the $m_p$-th index set of $A$ with respect to~the prime $p$. 
\end{Cor}

\begin{Exa}
We continue Example~\ref{ex:genset}, and determine the rings $\Int{A,\Matr{3}{\Z}}$ of integer-valued polynomials on $A$ and $\IntIm(A,\Matr{3}{\Z})$ of integer-valued images for 
\begin{align*}
 A = %
   \begin{pmatrix}
       4 &    0  & 0  \\
       0  &  16  & 0  \\
       0  &   0  & 32 \\
   \end{pmatrix}
\in \Matr{3}{\Z}.
\end{align*}
We know that 
\begin{align*}
 \Int{A,\Matr{3}{\Z}} = \sum_{p\in \Primes}\sum_{\ell\in\N} \frac{1}{p^{\ell}}\, \NullIovD{p^{\ell}}(A).
\end{align*}
We can use the data of Example~\ref{ex:genset} in order to conclude that 
\begin{align*}
 \Int{A,\Matr{3}{\Z}} &= \mu_A\Q[X] + \Z[X]+ \frac{1}{3}\,\NullIovD{3}(A) + \frac{1}{7}\,\NullIovD{7}(A) + \frac{1}{64}\,\NullIovD{64}(A)\\
&= \mu_A\Q[X] + \Z[X]+ \sum_{p\in \{2,3,7\}} \frac{1}{p^{m_p}}\; \NullIovD{p^{m_p}}(A) 
\end{align*}
where $m_2 = 6$ and $m_3=m_7=1$. Similarly to the computation in Example~\ref{ex:genset} it follows that the $\{0,1,\ell\}$ is the $\ell$-th index set of $A$ with respect to~7 (for $\ell\geq 1$) and $(X-4)(X-16)$ is a $(7)$-minimal polynomial of $A$ (since $4,16,32,\ldots$ is a $7$-ordering of $\{4,16,32\}$, cf.~Example~\ref{ex:genset}).
Hence
\begin{align*}
 \Int{A,\Matr{3}{\Z}} = &(X-4)(X-16)(X-32)\Q[X] + \Z[X] \\
                        &+ \frac{1}{3}(X-4)(X-32)\Z[X] + \frac{1}{7}(X-4)(X-16)\Z[X] \\
                        &+ \frac{1}{64}(X-4)(X-16)\Z[X] + \frac{1}{4}(X-4)\Z[X].
\end{align*}
And finally, this implies
\begin{align*}
 \Int{A,\Matr{3}{\Z}}(A) = \Z[A] &+%
   \begin{pmatrix}
       0 &    0  & 0  \\
       0  &  -64  & 0  \\
       0  &   0  & 0 \\
   \end{pmatrix}
   \Z[A]+%
   \begin{pmatrix}
       0 &   0  & 0  \\
       0  &  0  & 0  \\
       0  &  0  & 64 \\
   \end{pmatrix}
   \Z[A]\\ &+%
   \begin{pmatrix}
       0 &  0  & 0  \\
       0 &  0  & 0  \\
       0 &  0  & 7 \\
   \end{pmatrix}
   \Z[A]+%
   \begin{pmatrix}
       0 &   0  & 0  \\
       0 &   3  & 0  \\
       0 &   0  & 7 \\
   \end{pmatrix}
   \Z[A].
\end{align*}

\end{Exa}

\bibliographystyle{abbrv}
\bibliography{biblio}
\end{document}